\documentclass{amsart}
\usepackage{stmaryrd,amsmath,amssymb}
\usepackage{tikz,fullpage}
\usetikzlibrary{shapes,backgrounds}

\begin{document}

\newcommand{\zero}{\mathbf{0}}
\newcommand{\one}{\mathbf{1}}
\newcommand{\BC}[1]{\langle #1\rangle}

\newtheorem{obs}{Observation}
\newtheorem{lem}[obs]{Lemma}
\newtheorem{cor}[obs]{Corollary}
\newtheorem{prop}[obs]{Proposition}
\newtheorem{thm}[obs]{Theorem}
\newtheorem{rem}[obs]{Remark}

\newtheorem{De}[obs]{Definition}
\newtheorem{Ex}[obs]{Example}

\title{Inhabitants of interesting subsets of the Bousfield lattice}

\author{Andrew Brooke-Taylor$^1$, Benedikt L\"owe$^{2,3,4}$ \&
Birgit Richter$^{3}$}

\address{$^1$
School of Mathematics,
University of Leeds,
Leeds LS2 9JT,
United Kingdom}

\address{$^2$ Institute for Logic, Language and Computation,
Universiteit van Amsterdam, Postbus 94242, 1090 GE Amsterdam, The
Netherlands}

\address{$^3$ Fachbereich
 Mathematik, Universit\"at Hamburg,
Bundesstrasse 55, 20146 Hamburg, Germany}

\address{$^4$ Department of Pure Mathematics and Mathematical Statistics,
Christ's College, \& Churchill College, University of Cambridge,
Wilberforce Road, Cambridge CB3 0WB, United Kingdom}

\email{a.d.brooke-taylor@leeds.ac.uk}
\email{b.loewe@uva.nl}
\email{birgit.richter@uni-hamburg.de}
\date{\today}

\begin{abstract}
The set of Bousfield classes has some important subsets such as the
distributive lattice $\mathbf{DL}$ of all classes $\BC{E}$ which are
smash idempotent and the complete Boolean algebra $\mathbf{cBA}$ 
of closed classes. We provide examples of spectra that are
in $\mathbf{DL}$, but not in $\mathbf{cBA}$; in particular, for every
prime $p$, the Bousfield class of the Eilenberg-MacLane spectrum $\BC{H\mathbb{F}_p}$ is in
$\mathbf{DL}{\setminus}\mathbf{cBA}$. \end{abstract}

\thanks{
The first author acknowledges the financial support of the Research 
Networking Programme INFTY funded by the \textsl{European Science 
Foundation} (ESF),
the \textsl{Japan Society for the Promotion of Science} (JSPS) via a 
Postdoctoral Fellowship for Foreign Researchers and JSPS Grant-in-Aid 
\textsf{2301765},
and the \textsl{Engineering and Physical Sciences Research Council} 
(EPSRC) via the Early Career Fellowship \emph{Bringing Set Theory \& Algebraic
Topology Together} (\textsf{EP/K035703/1}).
The second author acknowledges financial support in the form of a VLAC 
fellowship-in-residence at the \textsl{Koninklijke Vlaamse Academie von 
Belgi\"e voor Wetenschappen en Kunsten} and an International Exchanges 
grant of the \textsl{Royal Society} (reference \textsf{IE141198}).
All three authors were Visiting Fellows of the research programme 
\emph{Semantics \& Syntax} (SAS) at the \textsl{Isaac Newton Institute 
for Mathematical Sciences} in Cambridge, England whilst part of this 
research was done.
}

\maketitle

\section{Introduction \& Definitions}

In the original paper \cite{bousfield} introducing the Bousfield lattice
$\mathbf{B}$, Bousfield also introduces its subsets $\mathbf{BA}$ and
$\mathbf{DL}$ and identifies the location of many explicit Bousfield
classes. In \cite[Definition 6.3]{hp1999}, Hovey and Palmieri add a
third interesting subset, denoted by $\mathbf{cBA}$. (We shall give
definitions below.) It is easy to see that
$$\mathbf{BA} \subseteq
\mathbf{cBA} \subseteq \mathbf{DL} \subseteq \mathbf{B}.$$
In this
paper, we deal with the question of which and how many spectra live in the
various parts of $\mathbf{B}$ defined by this chain of inclusions. The
main cardinality results of this paper (lower bounds) are graphically
represented as in Figure \ref{venn} and concern the dark grey parts.

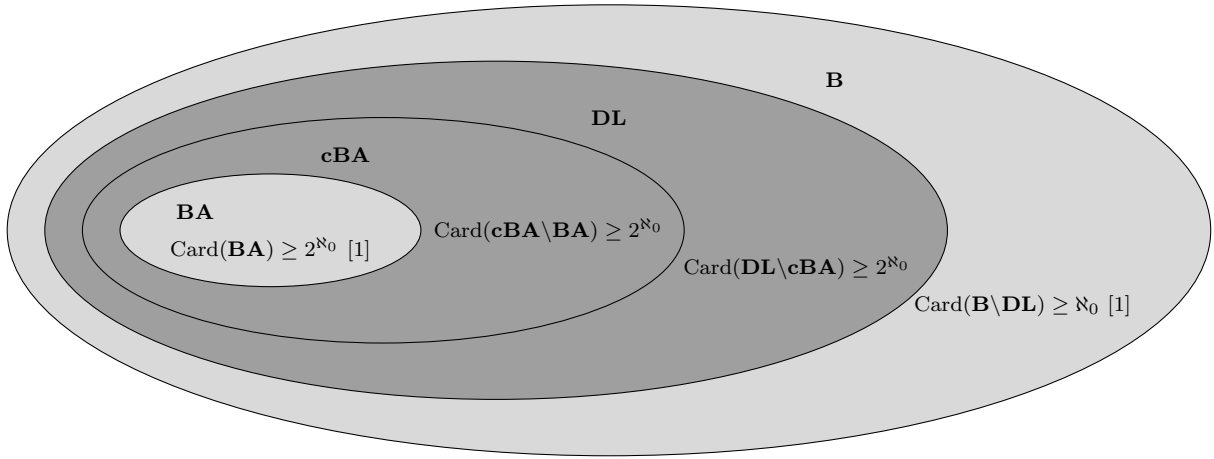
\begin{figure}[h]
\begin{center}
{\footnotesize
\begin{tikzpicture}

\draw[fill=gray!30] (-5,1)
ellipse (8 and 3);

\draw[fill=gray!75] (-6.5,1)
ellipse (6 and 2.25);

\draw[fill=gray!75] (-8,1)
ellipse (4 and 1.5);

\draw[fill=gray!30] (-9.5,1)
ellipse (2 and .75);

\node (BA) at (-10.5,1.25) {$\mathbf{BA}$};
\node (BA2) at (-9.5,0.75) {$\mathrm{Card}(\mathbf{BA})\geq2^{\aleph_0}$ \cite{bousfield}};

\node (cBA) at (-8.5,2) {$\mathbf{cBA}$};
\node (cBA2) at (-5.8,1)
{$\mathrm{Card}(\mathbf{cBA}{\setminus}\mathbf{BA})\geq 2^{\aleph_0}$};

\node (DL) at (-5,2.5) {$\mathbf{DL}$};
\node (DL2) at (-2.5,.5)
{$\mathrm{Card}(\mathbf{DL}{\setminus}\mathbf{cBA})\geq 2^{\aleph_0}$};

\node (B) at (-2,3) {$\mathbf{B}$};
\node (B2) at (.5,0)
{$\mathrm{Card}(\mathbf{B}{\setminus}\mathbf{DL})\geq \aleph_0$ \cite{bousfield}};

\end{tikzpicture}}
\end{center}
\caption{Lower bounds for the sizes of the four differences of subsets of
$\mathbf{B}$.}\label{venn}
\end{figure}

\section{Definitions}

In order to fix notation, we give the relevant definitions, following
closely the exposition in \cite{hp1999}.  We consider the Bousfield
equivalence of spectra \cite{bousfield}: two spectra $X$ and $Y$ are
equivalent if for all spectra $E$, $X_*(E) = 0$ if and only $Y_*(E)=0$
(alternatively put: $X\wedge E\simeq *$ if and only if $Y\wedge E\simeq *$).
For a spectrum $X$, we write $\BC{X}$ for the class of all spectra $E$
with $X_*(E)=0$.  The class of all Bousfield classes is denoted by
$\mathbf{B}$. By a theorem of Ohkawa \cite{ohkawa,dp2001}, it is known
that $\mathbf{B}$ is a set and
$$2^{\aleph_0} \leq
\mathrm{Card}(\mathbf{B}) \leq 2^{2^{\aleph_0}}.$$
This set is a poset
with respect to reverse inclusion: $\BC{X}\leq \BC{Y}$ if and only if
for all spectra $Z$, $Y_*Z= 0$ implies $X_*Z=0$. The poset
$(\mathbf{B},\leq)$ has a largest element $\one := \BC{S}$ where $S$ is
the sphere spectrum and we denote by $\zero$ the minimal element which
is the Bousfield class of the trivial spectrum.  
We work at a fixed but arbitrary prime $p$, \emph{i.e.}, we consider $p$-local 
spectra. 

For every prime $p$,
$K(n)$ denotes the $n$th Morava $K$-theory spectrum with coefficients
$\pi_*(K(n)) = \mathbb{F}_p[v_n^{\pm 1}]$ where the degree of $v_n$ is
$2p^n-2$. We use the convention that $K(\infty)$ is the mod $p$
Eilenberg-MacLane spectrum, $H\mathbb{F}_p$. For any subset $S
\subseteq \mathbb{N} \cup \{\infty\}$, we denote by $K(S)$ the spectrum
$\bigvee_{n \in S} K(n)$. 

The topological operations $\wedge$ and $\vee$ of taking smash products
and wedges, respectively, are well-defined on $\mathbf{B}$; the class
$\BC{\bigvee_{i\in I} X_i}$ is the least upper bound (``join'') in the 
structure
$(\mathbf{B},\leq)$ of the classes $\BC{X_i}$ \cite[(2.2)]{bousfield}, 
but in general, $\wedge$ does
not produce the greatest lower bound. We can define the greatest lower
bound (``meet'') by
$$\bigcurlywedge \mathcal{X} := \bigvee\{Z\,;\,\forall X\in\mathcal{X}(Z\leq X)\},$$
and observe that $\wedge$ and $\curlywedge$ can differ quite a bit: the
Brown-Comenetz dual $I$ of the $p$-local sphere spectrum satisfies
$\BC{I}\wedge\BC{I} = \zero \neq \BC{I} = \BC{I}\curlywedge\BC{I}$
\cite[Lemma 2.5]{bousfield}.

The complete lattice $(\mathbf{B},\curlywedge,\vee)$ is endowed with a pseudo-complementation function
$$aX := \bigvee\{Z\,;\,Z\wedge X = 0\}$$
which is well-defined on Bousfield classes, \emph{i.e.}, $a\BC{X} := \BC{aX}$
is independent of the choice of representative $X$ of $\BC{X}$.
The function $a$ is not in general a complement. While $a^2 =
\mathrm{id}$ and $a\BC{X}\wedge \BC{X} = \zero$, we may not have
$a\BC{X}\vee \BC{X} = \one$ \cite[Lemma 2.7]{bousfield}. Bousfield
defined two subclasses of $\mathbf{B}$ as follows:
\begin{align*}
\mathbf{BA} &:= \{\BC{X}\,;\,\BC{X}\vee a\BC{X} = \one\}\mbox{, and}\\
\mathbf{DL} &:= \{\BC{X}\,;\,\BC{X}\wedge\BC{X}=\BC{X}\}.
\end{align*}

Many examples for classes in $\mathbf{BA}$ or $\mathbf{DL}$ are known. 
Bousfield showed in \cite{bousfield} that every Moore spectrum of an 
abelian group is in $\mathbf{BA}$ and so are the periodic topological 
$K$-theory spectra $\BC{KO} = \BC{KU}$; furthermore, he shows that 
(arbitrary joins of) finite CW spectra also give classes in $\mathbf{BA}$. 
Every class of a 
ring spectrum is in $\mathbf{DL}$ but not necessarily in $\mathbf{BA}$ 
\cite[\S\;2.6]{bousfield}; in particular, all Eilenberg-MacLane spectra 
of rings are in $\mathbf{DL}$, but, \emph{e.g.}, the class of the 
Eilenberg-MacLane spectrum of the integers, $\BC{H\mathbb{Z}}$, is in 
$\mathbf{DL} {\setminus} \mathbf{BA}$ \cite[Lemma 2.7]{bousfield}. 
However, the Brown-Comenetz duals of ($p$-local) spheres are not in 
$\mathbf{DL}$ \cite[Lemma 2.5]{bousfield}.

We have that $\mathbf{BA}\subseteq\mathbf{DL}$; on $\mathbf{DL}$,
$\wedge$ and $\curlywedge$ coincide, and $(\mathbf{DL},\wedge,\vee)$ is
a distributive lattice. Furthermore, on $\mathbf{BA}$, $a$ is a true
complement, so $(\mathbf{BA},\wedge,\vee,\zero,\one,a)$ is a Boolean
algebra, but not complete.

There is a retraction from $\mathbf{B}$ to $\mathbf{DL}$ defined by 
$$r\BC{X} := \bigvee\{\BC{Z}\,;\,\BC{Z}\in\mathbf{DL}\mbox{ and 
}\BC{Z}\leq \BC{X}\}.$$ The pseudo-complementation function $a$ may not 
respect $\mathbf{DL}$, \emph{i.e.}, it could be that $\BC{X}\in\mathbf{DL}$, 
but $a\BC{X}\notin\mathbf{DL}$. On $\mathbf{DL}$, we therefore define a 
new pseudo-complement by
$$A\BC{X} := ra\BC{X}.$$
While $A^3 = A$ and $\BC{X} \leq A^2\BC{X}$, it is not in general the 
case that $A^2 = \mathrm{id}$. It is known \cite[Lemma~6.2(d)]{hp1999}
that $A$ converts joins to meets,
\emph{i.e.}, 
$$A(\bigvee\mathcal{X}) = \bigcurlywedge\{A\BC{X}\,;\,X\in\mathcal{X}\}.$$
Following \cite[Definition 6.3]{hp1999}, we define
$$\mathbf{cBA} := \{\BC{X} \in \mathbf{DL} \,;\,A^2\BC{X} = \BC{X}\}.$$
The set $\mathbf{cBA}$ carries a complete Boolean algebra structure \cite[Theorem 6.4]{hp1999}; however, it is not \linebreak $(\mathbf{cBA},\wedge,\vee,\zero,\one,A)$, but instead
$(\mathbf{cBA},\wedge,\curlyvee,\zero,\one,A)$ with $\curlyvee$ defined by
$$\bigcurlyvee \mathcal{X} := A^2\bigvee\mathcal{X}.$$

\section{Results}

We start with an observation on joins of elements in $\mathbf{BA}$ and 
use this to derive lower bounds for the size of 
$\mathbf{DL}{\setminus}\mathbf{cBA}$ and 
$\mathbf{cBA}{\setminus}\mathbf{BA}$.

\begin{lem}\label{lem:curly}
If $\mathcal{X}\subseteq\mathbf{BA}$, then $\bigcurlyvee\mathcal{X} = \bigvee\mathcal{X}$. In particular,
$\bigvee\mathcal{X} \in \mathbf{cBA}$.
\end{lem}

\begin{proof}
We have that
$$\bigcurlyvee\mathcal{X} = A^2\bigvee\mathcal{X} = rara\bigvee\mathcal{X},$$
and as $a$ converts joins to meets, the latter is equal to
$$rar\bigcurlywedge \{a\BC{X}\,;\,\BC{X}\in \mathcal{X}\}.$$
Since every $a\BC{X}$ is in $\mathbf{BA}$, it is also in $\mathbf{DL}$, and as 
$\mathbf{DL}$ is complete, 
$$\Xi := \bigcurlywedge \{a\BC{X}\,;\,\BC{X}\in \mathcal{X}\} \in\mathbf{DL}$$
and hence $r\Xi = \Xi$. 
Therefore, as $a$ sends meets to joins,
\begin{align*}
rar\Xi & = ra\Xi\\
&= r\bigvee \{a^2\BC{X}\,;\,\BC{X}\in\mathcal{X}\}\\
&= r \bigvee\{\BC{X}\,;\,\BC{X}\in\mathcal{X}\}\\
&= \bigvee\mathcal{X}.
\end{align*}
\end{proof}

\begin{prop} \label{prop:kinfinite}
If $S \subseteq \mathbb{N}$ is infinite, then $\BC{K(S)} =
\bigvee_{i\in S} \BC{K(i)} \in \mathbf{cBA}{\setminus}\mathbf{BA}$ and
$\BC{K(S)} \geq \BC{I}$.
\end{prop}

\begin{proof}
By Lemma \ref{lem:curly}, $\BC{K(S)}$ is in $\mathbf{cBA}$. Hovey showed
\cite[Proof of Theorem 3.6]{hovey:chrom} that the mod-$p$ Moore
spectrum, $M(p)$ is $K(S)$-local, so in particular $K(S)$ has a finite
local and \cite[Proposition 7.2]{hp1999} gives that $\BC{K(S)} \geq
\BC{I}$. If $K(S)$ were in $\mathbf{BA}$, having a finite local implies
\cite[Lemma 7.9]{hp1999} that $\BC{K(S) \wedge I} \neq \zero$. But we
know that $\BC{K(n) \wedge I} = \zero$ and hence using distributivity we
get that $\BC{K(S) \wedge I} = \zero$.
\end{proof}

\begin{cor}\label{cor:2}

We have a proper inclusion $\mathbf{BA}\subsetneqq\mathbf{cBA}$; in
fact, the set $\mathbf{cBA}{\setminus}\mathbf{BA}$ has size continuum.

\end{cor}

\begin{proof} Because $\mathbf{BA}$ is a Boolean algebra, 
$a\BC{X}\in\mathbf{BA}$ for elements 
$\BC{X}\in\mathbf{BA}\subseteq\mathbf{DL}$. Therefore, $A\BC{X} = 
ra\BC{X} = a\BC{X}$. But $a^2 = \mathrm{id}$, so ``$\subseteq$'' holds. 
For the non-equality, if $S \neq S'$ are infinite subsets of 
$\mathbb{N}$, then Dwyer and Palmieri showed that $\BC{K(S)} \neq 
\BC{K(S')}$ \cite[Lemma 3.4]{dp2001}, so there are continuum many 
elements in the complement. \end{proof}

To sum up, we have
$$\mathbf{BA}\subsetneqq \mathbf{cBA} \subseteq \mathbf{DL} \subsetneqq \mathbf{B}.$$
Hovey and Palmieri argue that the middle inclusion is also proper:
\begin{quote}
This argument also implies that $A^2$ is not the identity---indeed, if
$A^2$ were the identity, one can check that $A$ would have to convert
meets to joins. However, we do not know a specific spectrum $X$ in
$\mathbf{DL}$ for which $A^2\BC{X} \neq \BC{X}$. \cite[p.\ 185]{hp1999}
\end{quote}

We analyse the argument sketched in the above quote:

\begin{lem}\label{lem:andrew}
Let
$\mathcal{X}\subseteq\mathbf{DL}$ be any set such that 
$A^2$ is the identity for each $\BC{X}\in\mathcal{X}$ and for
$\bigvee\{A\BC{X}\,;\,\BC{X}\in\mathcal{X}\}$. Then
$$A(\bigcurlywedge\mathcal{X}) = \bigvee \{A\BC{X}\,;\,\BC{X}\in\mathcal{X}\}.$$
\end{lem}

\begin{proof}
Since $A$ converts joins to meets, under the assumption 
of the lemma, we have
\begin{align*}
A(\bigcurlywedge\mathcal{X})
& = A\bigcurlywedge\{A^2\BC{X}\,;\,\BC{X}\in\mathcal{X}\}\\
& = A^2\bigvee \{A\BC{X}\,;\,\BC{X}\in\mathcal{X}\}\\
& = \bigvee \{A\BC{X}\,;\,\BC{X}\in\mathcal{X}\}.
\end{align*}
\end{proof}

\begin{cor}[Hovey-Palmieri] \label{cor:HP}
The operation $A^2$ is not the identity on $\mathbf{DL}$; \emph{i.e.},
$\mathbf{cBA}\subsetneqq \mathbf{DL}$.
\end{cor}

\begin{proof} Let $X := K(\mathbb{N})$, $Y := H\mathbb{F}_p = 
K(\infty)$, and $\mathcal{X} := \{X,Y\} \subseteq \mathbf{DL}$. We 
assume towards a contradiction that $A^2$ is the identity on 
$\mathbf{DL}$, so in particular, the assumptions of Lemma 
\ref{lem:andrew} are satisfied for $\mathcal{X}$. But 
$\BC{X}\curlywedge\BC{Y} = \BC{X}\wedge\BC{Y} = \zero$, hence 
$A(\BC{X}\curlywedge\BC{Y}) = \one$. On the other hand, $A\BC{X}\vee 
A\BC{Y} \leq a\BC{I} < \one$, in contradiction to Lemma 
\ref{lem:andrew}. \end{proof}

The proof of Corollary \ref{cor:HP} due to Hovey and Palmieri yields a 
trichotomy result: at least one of $\BC{K(\mathbb{N})}$, 
$\BC{H\mathbb{F}_p}$, and $A\BC{K(\mathbb{N})}\vee A\BC{H\mathbb{F}_p}$ 
is not in $\mathbf{cBA}$. We improve this in our Dichotomy Lemma 
\ref{lem:dich} to a dichotomy which will allow us to identify concrete 
elements in $\mathbf{DL}{\setminus}\mathbf{cBA}$.

\begin{lem}\label{lem:zeroone}
For any spectrum, the condition $A\BC{E} < \one$ is equivalent to
$\BC{E} \neq \zero$.
\end{lem}
\begin{proof}
If $\BC{E}=\zero$, then clearly $A\BC{E} = \one$. Conversely, if 
$A\BC{E} = \one$, then $a\BC{E} \geq A\BC{E} = \one$, and so 
$$\BC{E} = \one \wedge \BC{E} = a\BC{E} \wedge \BC{E} = \zero.$$
\end{proof}

\begin{lem}[Dichotomy Lemma]\label{lem:dich}
Let $X$ and $Y$ be spectra, and let $E$ be a spectrum
such that $\BC{E} \neq \zero$. Suppose that the following conditions hold:
\begin{enumerate}
\item $\BC{X}\in\mathbf{DL}$,
\item $\BC{Y}\in\mathbf{DL}$,
\item $\BC{X}\wedge\BC{Y} = \zero$,
\item $\BC{E}\leq \BC{X}$, and
\item $\BC{E}\leq \BC{Y}$.
\end{enumerate}
Then $\BC{X}$ or $\BC{Y}$ is not in $\mathbf{cBA}$.
\end{lem}

Note that conditions (4) and (5) are equivalent to saying that $\BC{X} 
\curlywedge \BC{Y} \neq \zero$, and thus the Dichotomy Lemma extracts 
the failure of $A^2 = \mathrm{id}$ from the discrepancy between 
$\curlywedge$ and $\wedge$ in $\mathbf{B}$.

\begin{proof}
Assume that $A^2\BC{X} = \BC{X}$ and $A^2\BC{Y} = \BC{Y}$.
Since $A$ converts joins to meets, we get by our assumption on
$X$ and $Y$
$$\one = A\zero = A(\BC{X}\wedge \BC{Y})
= A(A^2\BC{X} \wedge A^2\BC{Y}) = A^2(A\BC{X} \vee A\BC{Y})$$
and the latter is $A\BC{X} \curlyvee A\BC{Y}$ by definition of $\curlyvee$.
As $A$ is order-reversing we get $A\BC{X} \leq A\BC{E}$ and $A\BC{Y} \leq
A\BC{E}$ and hence (using Lemma \ref{lem:zeroone})
$$ \one = A^2(A\BC{X} \vee A\BC{Y})
= A\BC{X} \curlyvee A\BC{Y} \leq A\BC{E} \curlyvee  A\BC{E}
= A\BC{E} < \one,$$
a contradiction, showing that our assumption that both $\BC{X}$ and $\BC{Y}$
are in $\mathbf{cBA}$ cannot hold.
\end{proof}

As usual, we call a set 
$S \subset \mathbb{N} \cup \{\infty\}$ \emph{coinfinite},
if its
complement $(\mathbb{N} \cup \{\infty\}) {\setminus} S$ is infinite.

\begin{thm}\label{thm:notcBA}
For any coinfinite set
$S \subseteq \mathbb{N} \cup \{\infty\}$ with 
$\infty \in S$, we have that 
$\BC{K(S)}$ is not in $\mathbf{cBA}$.
\end{thm}

\begin{proof}
In Lemma \ref{lem:dich}, choose $E$ to be the Brown-Comenetz dual of the
$p$-local sphere spectrum, $I$. We know by \cite [Lemma 7.1(c)]{hp1999}
that  $\BC{H\mathbb{F}_p}\geq \BC{I}$, and hence 
$\BC{K(S)}\geq\BC{I}$.
As the complement $\overline{S}:=(\mathbb{N}\cup\{\infty\}){\setminus}S$ is
infinite,  we get by Proposition \ref{prop:kinfinite} that
$\BC{K(\overline{S})} \geq \BC{I}$.
Both, $\BC{K(S)}$   and
$\BC{K(\overline{S})}$ are in $\mathbf{DL}$ and $\BC{K(S)} \wedge \BC{K(\overline{S})} = \zero$.
Thus all conditions of the Dichotomy Lemma are satisfied, and we get that
one of $\BC{K(S)}$ and $\BC{K(\overline{S})}$ is not in $\mathbf{cBA}$. 
However, by Corollary \ref{cor:2}, 
$\BC{K(\overline{S})}\in\mathbf{cBA}$, so $\BC{K(S)}\in\mathbf{DL} 
{\setminus} \mathbf{cBA}$. \end{proof}

\begin{cor} There are at least $2^{\aleph_0}$ Bousfield classes in 
$\mathbf{DL} {\setminus} \mathbf{cBA}$. \end{cor}

\begin{proof} This follows directly from Theorem \ref{thm:notcBA} and 
\cite[Lemma 3.4]{dp2001}, as there are $2^{\aleph_0}$ many coinfinite 
subsets of $\mathbb{N} \cup \{\infty\}$. \end{proof}

\section{Applications}

Several conjectures made by Hovey and Palmieri in \cite{hp1999} suggest 
that $\BC{H\mathbb{F}_p}$ is not in $\mathbf{cBA}$ \cite[Proposition 
6.14]{hp1999}. This follows directly from our Theorem \ref{thm:notcBA}:

\begin{cor} For every prime $p$, we have that $\BC{H\mathbb{F}_p}\in 
\mathbf{DL} {\setminus} \mathbf{cBA}$. \end{cor}

\begin{proof} This is clear from Theorem \ref{thm:notcBA}, as 
$\BC{H\mathbb{F}_p} = \BC{K(\infty)} = \BC{K(\{\infty\})}$ where 
$\{\infty\}$ is coinfinite in $\mathbb{N}\cup\{\infty\}$.\end{proof}

Our method also 
identifies several other explicit 
Bousfield classes in $\mathbf{DL} {\setminus} \mathbf{cBA}$.
The following examples exploit the fact that for any self-map of a spectrum
$X$, $f \colon \Sigma^{|f|}X \rightarrow X$ one gets by \cite[Lemma 1.34]{ravenel}
that
$$ \BC{X} = \BC{C_f} \vee \BC{X[f^{-1}]}. $$
Here, $C_f$ denotes the cofiber of $f$ and $X[f^{-1}]$ is the telescope.
Then the 
Bousfield
class of the Eilenberg-MacLane spectrum of the $p$-local integers,
$H\mathbb{Z}_{(p)}$, is $\BC{K(\{0,\infty\})}$.
This 
is a special case of a truncated Brown-Peterson spectrum
$BP\BC{n}$ with $\pi_*(BP\BC{n}) = \mathbb{Z}_{(p)}[v_1, \ldots, v_n]$
($|v_i|= 2p^i-2$). Multiplication by $v_n$ is a self-map on $BP\BC{n}$ with
cofiber $BP\BC{n-1}$ and $BP\BC{n}[v_n^{-1}] = E(n)$. An iteration then gives
(cf.\ \cite[Theorem 2.1]{ravenel})
$\BC{BP\BC{n}} = \BC{E(n)} \vee \BC{H\mathbb{F}_p}$.
As the Bousfield class of $E(n)$ is $\BC{K(0)} \vee \ldots \vee \BC{K(n)}$
we obtain
$\BC{BP\BC{n}} = \BC{K(\{0,\ldots,n,\infty\})}$.

\begin{cor} For every prime $p$ and every natural number $n$, we have 
that $\BC{H\mathbb{Z}_{(p)}}$ and $\BC{BP\BC{n}}$ are in $\mathbf{DL} 
{\setminus} \mathbf{cBA}.$ \end{cor}

\begin{proof}The subsets $\{0,\infty\}$ and 
$\{0,\ldots,n,\infty\}$ are coinfinite in $\mathbb{N}\cup\{\infty\}$. 
\end{proof}

For the connective Morava $K$-theory $k(n)$ (with $\pi_*k(n) =
\mathbb{F}_p[v_n]$) we get
$ \BC{k(n)} = \BC{K(n)} \vee \BC{H\mathbb{F}_p}
= \BC{K(\{n,\infty\})}$.

\begin{cor}
For every natural number $n$, 
$\BC{k(n)}\in\mathbf{DL} {\setminus}
\mathbf{cBA}$.
\end{cor}

\begin{proof}This follows from Theorem \ref{thm:notcBA}, as $\{n,\infty\}$ is
coinfinite in $\mathbb{N}\cup\{\infty\}$. 
\end{proof}

Similar to the Morava $K$-theory spectra $K(n)$ we can consider the telescopes
$T(n)$ of $v_n$-maps. (Cf.\ \cite[\S 5]{hp1999} for details.) It is known that
$$ \BC{T(n)} = \BC{K(n)} \vee \BC{A(n)}$$
where $A(n)$ is the spectrum describing the failure of the telescope
conjecture. We set $\BC{T(\infty)} = \BC{H\mathbb{F}_p}$.
The classes $\BC{T(n)}$ and $\BC{A(n)}$ are in $\mathbf{BA}$ but
$\bigvee_{\mathbb{N}} \BC{T(n)} \notin \mathbf{BA}$ by
\cite[Corollary 7.10]{hp1999}. By Lemma \ref{lem:curly}, we know that
for any $S\subseteq\mathbb{N}$, we have that 
$\bigvee_{n\in S} \BC{T(n)} \in \mathbf{cBA}.$
An argument similar to the proof of Proposition
\ref{prop:kinfinite} yields Proposition \ref{prop:tinfinite}.
\begin{prop} \label{prop:tinfinite}
If $S \subseteq \mathbb{N}$ is infinite, then $\BC{T(S)} =
\bigvee_{i\in S} \BC{T(i)} \in \mathbf{cBA}{\setminus}\mathbf{BA}$ and
$\BC{T(S)} \geq \BC{I}$.
\end{prop}

\begin{thm}
Let $S \subseteq \mathbb{N} \cup \{\infty\}$ be a coinfinite subset with
$\infty \in S$. Then $\BC{T(S)}$ is not in $\mathbf{cBA}$.
\end{thm}
\begin{proof}
Again, we use the Brown-Comenetz dual of the $p$-local sphere as $E$ in the
Dichotomy Lemma. Let $\overline{S}$ be the complement of $S$.
As $\BC{T(n)} \geq \BC{K(n)}$ and as $\infty \in S$ we have that
$$ \bigvee_{n \in S} \BC{T(n)} \geq \bigvee_{n \in S} \BC{K(n)} \geq \BC{I}$$
and $\bigvee_{n \in \overline{S}} \BC{T(n)} \geq \BC{I}$.
The telescopes satisfy $\BC{T(n)} \wedge \BC{T(m)} = \zero$ for $m \neq n$:
cf.\ \cite[\S 5]{hp1999} for the cases $n \neq \infty \neq m$ and cf.\ 
the proof of
\cite[Proposition 6.14]{hp1999} for $\BC{H\mathbb{F}_p)} \wedge \bigvee_{\mathbb{N}}
\BC{T(n)} = \zero$. Therefore we obtain that one of
$\bigvee_{n \in S} \BC{T(n)}$ or $\bigvee_{n \in \overline{S}} \BC{T(n)} $ cannot be an
element of $\mathbf{cBA}$, but $\bigvee_{n \in \overline{S}} \BC{T(n)}$ is in
$\mathbf{cBA}$ by Proposition \ref{prop:tinfinite}.
\end{proof}

\end{document}